\documentclass[12pt]{amsart}
\usepackage[T1]{fontenc}
\usepackage[english]{babel}
\usepackage{amscd,amsmath,amsthm,amssymb,graphics}
\usepackage{lmodern,pst-node}
\usepackage{pstricks}
\usepackage{subcaption}
\usepackage{tikz}

\usetikzlibrary{positioning}
\usepackage{tikz,fullpage}
\usetikzlibrary{arrows,%
	petri,%
	topaths, automata}%
\usepackage{etex}
\usetikzlibrary{decorations.markings}
\tikzstyle{vertex}=[circle, draw, inner sep=0pt, minimum size=6pt]

\usepackage{multicol}
\usepackage{epic,eepic}
\usepackage{amsfonts,amssymb,amscd,amsmath,enumitem,verbatim}
\psset{unit=0.7cm,linewidth=0.8pt,arrowsize=2.5pt 4}
\newpsstyle{fatline}{linewidth=1.5pt}
\newpsstyle{fyp}{fillstyle=solid,fillcolor=verylight}
\definecolor{verylight}{gray}{0.97}
\definecolor{light}{gray}{0.9}
\definecolor{medium}{gray}{0.85}
\definecolor{dark}{gray}{0.6}

\def\frk{\frak}               

\def\Phi{{\frk n}}
\def\Phi{{\frk N}}
\def\opn#1#2{\def#1{\operatorname{#2}}} 
\opn\chara{char} \opn\length{\ell} \opn\pd{pd} \opn\rk{rk}
\opn\projdim{proj\,dim} \opn\injdim{inj\,dim} \opn\rank{rank}
\opn\depth{depth} \opn\grade{grade} \opn\height{height}
\opn\embdim{emb\,dim} \opn\codim{codim}

\opn\Tr{Tr} \opn\bigrank{big\,rank}
\opn\superheight{superheight}\opn\lcm{lcm}
\opn\trdeg{tr\,deg}
	\opn\reg{reg} \opn\lreg{lreg} \opn\ini{in} \opn\lpd{lpd}
	\opn\size{size}\opn\bigsize{bigsize}
	\opn\cosize{cosize}\opn\bigcosize{bigcosize}
	\opn\sdepth{sdepth}\opn\sreg{sreg}
	\opn\link{link}\opn\fdepth{fdepth}
	\opn\deg{deg}
	\opn\max{max}
	\opn\indeg{indeg}
	\opn\min{min}
	\opn\psln{psln}
	\opn\div{div} \opn\Div{Div} \opn\cl{cl} \opn\Cl{Cl}
	\let\epsilon\varepsilon
	\let\phi=\varphi
	\let\kappa=\varkappa

	\opn\Spec{Spec} \opn\Supp{Supp} \opn\supp{supp} \opn\Sing{Sing}
	\opn\Ass{Ass} \opn\Min{Min}\opn\Mon{Mon} \opn\dstab{dstab} \opn\astab{astab}
	\opn\Syz{Syz}
	\opn\Ann{Ann} \opn\Rad{Rad} \opn\Soc{Soc}
	\opn\Im{Im}
	\opn\Ind{Ind}
	\opn\del{del}
	\opn\Ker{Ker} \opn\Coker{Coker} \opn\Am{Am}
	\opn\Hom{Hom} \opn\Tor{Tor} \opn\Ext{Ext} \opn\End{End}
	\opn\Aut{Aut} \opn\id{id}
	
	\opn\nat{nat}
	\opn\pff{pf}
	\opn\Pf{Pf} \opn\GL{GL} \opn\SL{SL} \opn\mod{mod} \opn\ord{ord}
	\opn\Gin{Gin} \opn\Hilb{Hilb}\opn\sort{sort}
	\opn\initial{init}
	\opn\ende{end}
	\opn\height{height}
	\opn\bight{bight}
	\opn\hte{ht}
	\opn\indeg{indeg}
	\opn\reg{reg}
	\opn\depth{depth}
	\opn\type{type}
	\opn\ldim{ldim}
	\opn\maxdeg{maxdeg}
	\opn\aff{aff} \opn\con{conv} \opn\relint{relint} \opn\st{st}
	\opn\lk{lk} \opn\cn{cn} \opn\core{core} \opn\vol{vol}
	\opn\link{link} \opn\star{star}\opn\lex{lex}
	\opn\gr{gr}
	
	\def\pot#1#2{#1[\kern-0.28ex[#2]\kern-0.28ex]}

	\opn\dirlim{\underrightarrow{\lim}}
	\opn\inivlim{\underleftarrow{\lim}}
		\let\to=\rightarrow
		
		\def\Implies{\ifmmode\Longrightarrow \else
			\unskip${}\Longrightarrow{}$\ignorespaces\fi}
		\def\implies{\ifmmode\Rightarrow \else
			\unskip${}\Rightarrow{}$\ignorespaces\fi}
		\def\iff{\ifmmode\Longleftrightarrow \else
			\unskip${}\Longleftrightarrow{}$\ignorespaces\fi}

		\let\:=\colon
		\theoremstyle{plain}
		\newtheorem{Theorem}{Theorem}[section]
		\newtheorem{Lemma}[Theorem]{Lemma}
		\newtheorem{Corollary}[Theorem]{Corollary}
		\newtheorem{Proposition}[Theorem]{Proposition}
		
		\newtheorem{Conjecture}[Theorem]{Conjecture}
		\newtheorem{Question}[Theorem]{Question}
		\theoremstyle{definition}
		\newtheorem{Definition}[Theorem]{Definition}

		\newtheorem{Example}[Theorem]{Example}

		\let\epsilon\varepsilon
		\let\kappa=\varkappa
		\def\qed{\ifhmode\textqed\fi
			\ifmmode\ifinner\quad\qedsymbol\else\dispqed\fi\fi}
		\def\textqed{\unskip\nobreak\penalty50
			\hskip2em\hbox{}\nobreak\hfil\qedsymbol
			\parfillskip=0pt \finalhyphendemerits=0}
		\def\dispqed{\rlap{\qquad\qedsymbol}}
		
		\opn\dis{dis}
		\def\pnt{{\raise0.5mm\hbox{\large\bf.}}}
		
		\opn\Lex{Lex}
		
\begin{document}
 \title{An upper bound on stability of powers of matroidal ideals}

 \author {Mozhgan Koolani, Amir Mafi* and Parasto Soufivand}

\address{Mozghan Kolani, Department of Mathematics, University of Kurdistan, P.O. Box: 416, Sanandaj,
Iran.}
\email{mozhgan.koolani@gmail.com}

\address{Amir Mafi, Department of Mathematics, University of Kurdistan, P.O. Box: 416, Sanandaj,
Iran.}
\email{a\_mafi@ipm.ir}

\address{Parasto Soufivand, Department of Mathematics, University of Kurdistan, P.O. Box: 416, Sanandaj,
Iran.}
\email{Parisoufivand@gmail.com}

\subjclass[2010]{13A15, 13A30, 13C15, 13F55.}

\keywords{Monomial ideals, polymatroidal ideals, associated primes and depth stability number.\\
* Corresponding author}

\begin{abstract}
Let $R=K[x_1,\ldots,x_n]$ be a polynomial ring in $n$ variables over a field $K$ and $I$ be a matroidal ideal of degree $d$.
Let $\astab(I)$ and $\dstab(I)$ be the smallest integers $l$ and $k$, for which $\Ass(I^l)$ and $\depth(R/I^k)$ stabilize, respectively.
In this paper, we show that $\astab(I),\dstab(I)\leq\min\{d,\ell(I)\}$, where $\ell(I)$ is the analytic spread of $I$. Furthermore, by a counterexample we give a negative answer to the conjecture of Herzog and Qureshi \cite{HQ} about stability of matroidal ideals.
\end{abstract}

\maketitle

\section*{Introduction}
Throughout this paper, we assume that $R=K[x_1,\ldots,x_n]$ is the polynomial ring in $n$ variables over a field $K$ with the maximal ideal $\frak{m}=(x_1,\ldots,x_n)$ and $I$ is a monomial ideal of $R$ such that $G(I)$ is the unique minimal monomial generators set of $I$.
Let $\Ass(I)$ be the set of associated prime ideals of $R/I$. Brodmann \cite{B1} showed that there exists an integer $l$ such that $\Ass(I^t)=\Ass(I^{l})$ for all $t\geq l$. The smallest such integer $l$ is
called the {\it index of Ass-stability} of $I$,  and denoted by $\astab(I)$. Moreover, $\Ass(I^{l})$ is called the stable set of associated
prime ideals of $I$ and it is denoted by $\Ass^{\infty}(I)$. Brodmann \cite{B} also showed that there exists an integer $k$ such that $\depth R/I^s=\depth R/I^{k}$ for all $s\geq k$. The smallest such integer $k$ is called the {\it index of depth stability} of $I$ and denoted by $\dstab(I)$.
Herzog and the second author \cite{HM} proved that if $n=3$ then, for any graded ideal $I$ of $R$, $\astab(I)=\dstab(I)$. Also, they showed that for $n=4$ the indices $\astab(I)$ and $\dstab(I)$, in general, are unrelated.

A monomial ideal $I$ generated in a single degree is {\it polymatroidal ideal} when it is satisfying in the following exchange property: for all monomials $u,v\in G(I)$ with $\deg_{x_i}(u)>\deg_{x_i}(v)$, there exists an index $j$ such that $\deg_{x_j}(v)>\deg_{x_j}(u)$ and $x_j(u/x_i)\in G(I)$ (see \cite{HH} or \cite{HH2}). A square-free polymatroidal ideal is called a {\it matroidal ideal}.

The polymatroidal ideals have interesting properties, we recall some of useful properties which is needed in this paper:
(1) The product of polymatroidal ideals is again polymatroidal (see \cite[Theorem 5.3]{CH}), in particular each power of a polymatroidal ideal is polymatroidal; (2) $I$ is a polymatroidal ideal if and only if $(I:u)$ is a polymatroidal ideal for all monomials $u$ (see \cite[Theorem 1.1]{BH}); (3) if $I$ is a matroidal ideal of degree $d$, then $\depth(R/I)=d-1$ and $\pd(R/I)=n-d+1$ (see \cite[Theorem 2.5]{C} or \cite[Theorem 2.3]{SM}); (4) according to \cite{HQ} and \cite{HRV}, every polymatroidal ideal satisfying the persistence property and non-increasing depth functions, that is, if $I$ is a polymatroidal ideal then, for all $l,k$, there is the following sequences: $\Ass(I^l)\subseteq\Ass(I^{l+1})$ and $\depth(R/I^{k+1})\leq\depth(R/I^k)$, these properties, in general, for all square-free monomial ideals has a negative answer (see \cite[Example 2.9]{MS}); (5) for every polymatroidal ideal $I$, we have $\astab(I),\dstab(I)<\ell(I)$, where $\ell(I)$ is the analytic spread of $I$  (see \cite[Theorem 4.1]{HQ});
(6) if $I$ is a polymatroidal ideal, then $\lim_{k\to\infty}\depth R/I^k= n-\ell(I)$ (see \cite[Corollary 3.5]{HRV}); (7) if $I$ is a matroidal ideal of degree $d$, then $\astab(I)=1$ if and only if $\dstab(I)=1$ (see \cite[Theorem 3.3]{MN}).

Herzog and Qureshi \cite{HQ} raised the following conjectured:
\begin{Conjecture}\label{Co}
Let $I$ be any polymatroidal ideal. Does $\astab(I)=\dstab(I)$?
\end{Conjecture}

This conjecture in the following cases has a positive answer: (1) if $I$ is a polymatroidal ideal of degree $2$, (\cite[Theorem 2.12]{KM}); (2) if $I$ is a matroidal ideal of degree $3$, (\cite[Corollary 3.13]{MN}); (3) if $I$ is a transversal polymatroidal ideal, \cite[Corollaries 4.6, 4.14]{HRV}; (4) if $I$ is a Veronese type polymatroidal ideal \cite[Corollary 5.7]{HRV}, then $\astab(I)=\dstab(I)$. This conjecture, in general, has a negative answer, see \cite[Examples 2.21, 2.22]{KM}. However, it is still remained for all matroidal ideals.

The following question asked by Karimi and the second author \cite{KM}:
\begin{Question}\label{Q}
Let $I$ be any matroidal ideal. Then can we have $\astab(I),\dstab(I)\leq d$, where $d$ is the degree of $I$?
\end{Question}

If $I$ is a matroidal ideal of degree $2$ or a matroidal ideal of degree $3$ with $\frak{m}\notin\Ass^{\infty}(I)$, then the question has a positive answer, (see \cite[Corollary 2.13]{KM} and \cite[Proposition 3.8]{MN}). Also, if $I$ is a transversal matroidal ideal or if $I$ is a Veronese type matroidal ideal, then again this question has a positive answer (see (\cite[Corollaries 4.6, 4.14, 5.7]{HRV}).

In this paper, by a counterexample, we show that the Conjecture \ref{Co} in general for all matroidal ideals has a negative answer. Moreover, we give an affirmative answer to the Question \ref{Q}, that is, $\astab(I),\dstab(I)\leq d$ for all matroidal ideals.

For each unexplained notion or terminology, we refer the reader to \cite{HH1}. The explicit examples were performed with the help of the computer algebra system Macaulay2 \cite{GS}.

\section{The results}
Throughout this section we assume that $I$ is a matroidal ideal of $R$ with $\gcd(I)=1$ and $\supp(I)=\{x_1,\ldots,x_n\}$, where $\supp(I)=\cup_{i=1}^t\supp(u_i)$ such that $\supp(u)=\{x_i| u=x_1^{a_1}\ldots x_n^{a_n}, a_i\neq 0\}$ and $\gcd(I)=\gcd(u_1,\ldots,u_t)$ for $G(I)=\{u_1,\ldots,u_t\}$.

\begin{Definition}\cite[Definition 3.1]{HQ}
Let $G(I)=\{u_1,\ldots,u_t\}$. Then {\it linear relation graph} $\Gamma$ of $I$ is the graph with edge set
$E(\Gamma)=\{\{i,j\}:$ there exist $u_k,u_l\in G(I)$ such that $x_iu_k=x_ju_l\}$ and vertex set $V(\Gamma)=\cup_{\{i,j\}\in E(\Gamma)}\{i,j\}$.
\end{Definition}

Following \cite{HQ}, let $G$ be a finite simple graph with vertex set $V(G)$ and edge set $E(G)$. We denote by $c(G)$ the number of connected components of $G$, and for a subset $T\subset V(G)$ we denote by $G_T$ the graph restricted to the vertex set $T$. In particular, for each vertex $v\in V(G)$ we set $G\setminus v=G_{V(G)\setminus\{v\}}$. A vertex $v$ of $G$ is called a {\it cutpoint} if $c(G)<c(G\setminus v)$. A connected graph with no cutpoints is called {\it biconnected}. We have the following fact from graph theory:

A maximal biconnected subgraph of $G$ is called a biconnected component of $G$. Let $G_1,\dots,G_t$ be the biconnected components of $G$. Then $G=\cup_{i=1}^tG_i$ and each two distinct biconnected components intersect at most in a cutpoint. If $G$ is a biconnected, then each two distinct edges belong to a cycle.

\begin{Lemma}\label{L1}
Let $G$ be a graph and $I$ be a matroidal ideal attached to the graph matroid of $G$. Then $m\in\Ass^{\infty}(I)$ if and only if the linear relation graph $\Gamma$ of $I$ is a complete graph.
\end{Lemma}

\begin{proof}
$(\Longleftarrow)$. Suppose that the linear relation graph $\Gamma$ of $I$ is a complete graph. Then by \cite[Proposition 4.4]{HQ} we have $\ell(I)=\mid E(G)\mid$. Since $I$ is a matroidal ideal attached to the graph matroid of $G$, it follows $\mid E(G)\mid=n$. Therefore by \cite[Corollary 1.6]{HQ} we conclude $m\in\Ass^{\infty}(I)$.\\
$(\Longrightarrow)$. Let $G$ be a graph and $I\subset K[x_i| e_i\in E(G)]$ be a matroidal ideal attached to the graph matroid of $G$ such that $m\in\Ass^{\infty}(I)$. Then $G$ is a biconnected graph and $I$ is generated by monomials $u_F=\prod_{e_i\in E(F)}x_i$, where $F$ is a spanning forest of $G$, i.e., a subgraph of $G$ which is a forest and for which $V(F)=V(G)$. From the definition of $\Gamma$ we have $\{i,j\}$ is an edge of $\Gamma$ if there exists a spanning forest $F$ with $e_i\in E(F)$ such that $(E(F)\setminus\{e_i\})\cup\{e_j\}$ is also a spanning forest of $G$. If $e_i,e_j\in E(G)$, then there is a cycle $C$ in $G$ which is $e_i,e_j\in E(C)$. Therefore we can construct a spanning forest $F$ of $G$ such that $E(C)\setminus\{e_i\}$ is contained in $E(F)$ and $(E(F)\setminus\{e_i\})\cup\{e_j\}$ is again a spanning tree of $G$. Therefore $\{i,j\}\in E(\Gamma)$ and so $\Gamma$ is a complete graph, as required.
\end{proof}

\begin{Lemma}\label{L3}
Let $I$ be a matroidal ideal of degree $d$ and $A_i=\{u\in G(I)| \deg_{x_i}(u)=1\}$. Then $|A_i|\geq d$.
\end{Lemma}

\begin{proof}
By \cite[Lemma 2.16]{KM} we may assume that $d\leq n-2$.
Now, we use induction on $d$. If $d=2$ , then by definition of matroidal it is clear that $\mid A_i\mid\geq 2$. Suppose that the result has been proved for $d-1$. Now, we show that the result holds for $d$. By contrary, we may assume that $x_1$ is a variable such that $\mid A_1\mid<d$. Since $(I:x_1)$ is a matroidal ideal of degree $d-1$ in $K[x_{i_1},\ldots,x_{i_t}]$ where $t\leq n-1$ and also $\gcd(I:x_1)=1$, it follows that $d-1<t$. On the other hand, since $\mid A_1\mid<d$ we have $t\leq d-1$. This is a contradiction. Therefore $|A_i|\geq d$, as required.
\end{proof}

From now on we use $\Gamma_I$ as the linear relation graph $\Gamma$ of $I$.
\begin{Theorem}\label{T1}
Let $I$ be a matroidal ideal of degree $d$. If $\frak{m}\in\Ass^{\infty}(I)$, then $\depth R/I^d=0$. In particular, $\dstab(I)\leq d$.
\end{Theorem}

\begin{proof}
Suppose $\frak{m}\in\Ass^{\infty}(I)$. Therefore, by Lemma \ref{L1}, it follows $\Gamma_I$ is a connected complete graph. Since for every $1\leq i\neq j\leq n$, we have $\{i,j\}\in \Gamma_I$ and also by Lemma \ref{L3} we may assume that $|A_{i_1}|\geq |A_{i_2}|\geq\ldots\geq |A_{i_{2d-1}}|\geq |A_{i_j}|\geq d$ for all $2d-1<j\leq n$. We may set $u=(x_{i_1}\ldots x_{i_d})^{d-1}x_{i_{d+1}}\ldots x_{i_{2d-1}}$. If $d+t=n$ for some $1\leq t\leq d-1$, then we consider $x_{i_{d+t}}=\ldots=x_{i_{2d-1}}=x_n$. Since $I$ is the matroidal ideal attached to the graphic matroid of $G$ is generated by the monomials $u_F=\prod_{e_i\in E(F)}x_i$, where $F$ is a spanning forest of $G$, it follows that $x_{i_1}\ldots x_{i_d}\in I$. Again, by definition of $\Gamma_I$, we have $\{i,j\}\in \Gamma_I$ for every $1\leq i\neq j\leq n$ and so for the spanning forest $F$ with $e_{i_t}\in E(F)$ for all $1\leq t\leq d$ we have $(E(F)\setminus\{e_{i_t}\})\cup\{e_{i_j}\}$ is also a spanning forest of $G$ for all $1\leq j\leq n$. Thus $(x_{i_1}\ldots\widehat{x_{i_t}}\ldots x_{i_d})x_{i_j}\in I$ for all $1\leq t\leq d$ and $1\leq j\leq n$ and so $ux_j\in I^d$ for all $1\leq j\leq n$. Since $u\notin I^d$, it therefore follows that $I^d:u=\frak{m}$ and so $\frak{m}\in\Ass(I^d)$ and $\depth R/I^d=0$. Hence $\dstab(I)\leq d$, as required.
\end{proof}

\begin{Example}
For $n=6$, the matroidal ideal
$$I=(x_1x_3,x_1x_4,x_1x_5,x_1x_6,x_2x_3,x_2x_4,x_2x_5,x_2x_6,x_3x_5,x_3x_6,x_4x_5,x_4x_6)$$ constructed in \cite{HH1}.
Then $(I^2:x_1x_3x_5)=\frak{m}$.
\end{Example}

The monomial localization of a monomial ideal $I$ with respect to a monomial prime ideal $\frak{p}$ is the monomial ideal $I(\frak{p})$ which is obtained from $I$ by substituting the variables $x_i\notin\frak{p}$ by $1$. If $I$ is a square-free monomial ideal, then it is clear that $I(\frak{p})=I:(\prod_{x_i\notin{\frak{p}}}x_i)$. Also, for a monomial ideal $I$ of $R$ and for each $1\leq i\leq n$, we set $I[i]=(x^t[i]: x^t\in I)$ where $x^t[i]$ is the monomial $x_1^{t_1}\ldots\widehat{x_i^{t_i}}\ldots x_n^{t_n}$ such that the term $x_i^{t_i}$ of $x^t$ is omitted.

\begin{Corollary}\label{C1}
Let $I$ be a matroidal ideal of degree $d$. Then $\astab(I)\leq d$. In particular, if $\frak{m}\notin\Ass^{\infty}(I)$, then $\astab(I)\leq d-1$.
\end{Corollary}

\begin{proof}
Suppose $\frak{p}\in\Ass^{\infty}(I)$. Then $\frak{p}R_{\frak{p}}\in\Ass^{\infty}(I(\frak{p}))$.\\
Hence by Theorem \ref{T1} $\depth R_{\frak{p}}/I(\frak{p})^{\deg(I(\frak{p}))}=0$. Since $\deg(I(\frak{p}))\leq\deg(I)=d$, it follows that
$\depth R_{\frak{p}}/I(\frak{p})^d=0$. Therefore $\frak{p}R_{\frak{p}}\in\Ass(I^d(\frak{p}))$ and so $\frak{p}\in\Ass(I^d)$. Thus $\Ass^{\infty}(I)=\Ass(I^d)$ and so $\astab(I)\leq d$.
Suppose $\frak{m}\notin\Ass^{\infty}(I)$. Then we have $\Ass(I^d)=\cup_{i=1}^n\Ass(I[i]^d)$, (see \cite[Remark]{T} or \cite[Remark 2.6]{KM}).
From the first proof we conclude that $\astab(I[i])\leq d-1$ for all $1\leq i\leq n$. It therefore follows $\astab(I)\leq d-1$, as required.
\end{proof}

One of the most distinguished polymatroidal ideals is the ideal of Veronese type. For the fixed positive integers $d$ and $1\leq a_1\leq\ldots\leq a_n\leq d$. The ideal of Veronese type of $R$ indexed by $d$ and $(a_1,\ldots,a_n)$ is the ideal $I_{(d; a_1,\ldots,a_n)}$ which is generated by those monomials $u=x_1^{t_1}\ldots x_n^{t_n}$ of $R$ of degree $d$ with $t_j\leq a_j$ for each $1\leq j\leq n$. Note that if $a_j=1$ for all $j$, then $I_{(d; 1,\ldots,1)}$ is a Veronese type matroidal ideal.

\begin{Lemma}\label{L4}
Let $I$ be a matroidal ideal of degree $d$. Then $I\subseteq I_1\cap\ldots\cap I_s$, where $I_1,\ldots,I_s$ are Veronese type matroidal ideals with $\deg(I_i)=d_i$ for $1\leq i\leq s$ and $d=d_1+\ldots+d_s$ and also generated by the set of product of vertices of the connected components $\Gamma_1,\ldots,\Gamma_s$ of $\Gamma_I$.
\end{Lemma}

\begin{proof}
By \cite[Proposition 2.3]{HV}, we may assume that $s\leq d$ and $V(\Gamma_I)=\{1,\ldots,n\}$. If $\frak{m}\in\Ass^{\infty}(I)$ then, by Lemma \ref{L1},   $\Gamma_I$ is a complete graph and it is clear that $I$ is a subset of the Veronese type matroidal ideal of degree $d$. Hence we may assume that $\frak{m}\notin\Ass^{\infty}(I)$  and so $s\geq 2$ by \cite[Lemma 4.2]{HQ}. Now, suppose that there exists $1\leq i\leq s$ such that $I\not\subset I_i$. For simplicity, we may assume that $i=1$ and $\supp(I_1)=\{x_1,\ldots,x_t\}$, where $t<n$. Thus there are monomials $u,v\in G(I)$ such that $x_1\mid u$ and $\supp(v)\cap\{x_1,\ldots,x_t\}=\emptyset$. Therefore there exists $j>t$ such that $x_j\mid v$ and $x_j\nmid u$. By exchange property between $u,v$ we have $w=(u/x_1)x_j\in G(I)$ and so $x_1w=x_ju$. Hence $\{1,j\}\in E(\Gamma_1)$ and this is a contradiction since $j\notin V(\Gamma_1)$. Therefore $I\subseteq I_1\cap\ldots\cap I_s$, as required.
\end{proof}

\begin{Proposition}\label{P1}
Let $I$ be a matroidal ideal. Then we we have the following:
\begin{itemize}
\item[(1)] $\min\{\depth R/I^t\mid t\geq 1\}=s-1$, where $s$ is the number of connected component of $\Gamma_I$.
\item[(2)] $\dstab(I)=\min\{t\geq 1\mid \depth R/I^t=s-1\}$.
\end{itemize}
\end{Proposition}

\begin{proof}
By \cite[Proposition 2.3]{HV} and \cite[Lemma 4.2]{HQ}, we have $V(\Gamma_I)=\{1,\ldots,n\}$ and $\ell(I)=n-s+1$. Now by using \cite[Corollary 3.5]{HRV} we have $\min\{\depth R/I^t\mid t\geq 1\}=s-1$.
By definition and \cite[Corollary 3.5]{HRV} we have $\dstab(I)=\min\{t\geq 1\mid \depth R/I^t=n-\ell(I)\}$. Again by \cite[Lemma 4.2]{HQ} we conclude that $\dstab(I)=\min\{t\geq 1\mid \depth R/I^t=s-1\}$, as required.
\end{proof}

\begin{Theorem}\label{T2}
Let $I$ be matroidal ideal of degree $d$ such that $\frak{m}\notin\Ass^{\infty}(I)$. Then $\dstab(I)\leq d-1$.
\end{Theorem}

\begin{proof}
Since $\frak{m}\notin\Ass^{\infty}(I)$, by Lemma \ref{L4} we have $I\subseteq I_1\cap\ldots\cap I_s$ such that $I_1,\ldots,I_s$ are Veronese type matroidal ideals with $\deg(I_i)=d_i$ for $1\leq i\leq s$ and $d=d_1+\ldots+d_s$ also generated by the set of product of vertices of the connected components $\Gamma_1,\ldots,\Gamma_s$ of $\Gamma_I$. It is clear that $I\subseteq \prod _{j=1}^sI_j$ and so by exchange property for each $u\in G(I)$ we have $u=\prod_{j=1}^su_j$ such that $u_j\in G(I_j)$. Therefore we may consider that $I=\prod_{j=1}^sJ_j$ such that $J_j\subseteq I_j$ for all $1\leq j\leq s$. Thus by using \cite[Lemma 2.2]{HT} and Theorem \ref{T1} we have $\depth R/I^l=s-1$ for some $l<d$. Therefore $\dstab(I)\leq d-1$, as required.
\end{proof}

The following main result immediately follows by Theorems \ref{T1}, \ref{T2}, Corollary \ref{C1} and \cite[Theorem 4.1]{HQ}.
\begin{Corollary}\label{C2}
Let $I$ be a matroidal ideal of degree $d$. Then $\astab(I),\dstab(I)\leq \min\{d,\ell(I)\}$.
\end{Corollary}

\begin{Corollary}\label{C3}
Let $I$ be a matroidal ideal of degree $4$ such that $\frak{m}\notin\Ass^{\infty}(I)$. Then $\astab(I)=\dstab(I)$.
\end{Corollary}

\begin{proof}
By using the proof of Theorem \ref{T2} we may consider $I=J_1J_2$ such that $J_1, J_2$ are disjoint matroidal ideal such that $\deg(I_1)+\deg(I_2)=4$.
Therefore for all positive integer $l$, we readily conculde that  $\Ass(I^l)=\Ass(J_1^l)\cup\Ass(J_2^l)$. This yields that
${\astab(I)=\max\{\astab(J_1),\astab(J_2)\}}$. By \cite[Theorem 2.12]{KM} and \cite[Corollary 3.13]{MN}, we can conclude that $\astab(J_1)=\dstab(J_1), \astab(J_2)=\dstab(J_2)$ and also by \cite[Lemma 2.2]{HT} we have $\dstab(I)=\max\{\dstab(J_1),\dstab(J_2)\}$. It therefore follows $\astab(I)=\dstab(I)$, as required.
\end{proof}

By the following example, Karimi and Mafi \cite[Example 2.21]{KM} show that the Conjecture \ref{Co} for all polymatroidal ideals does not hold.
\begin{Example}
Let $n=4$ and \[
I=(x_1x_2x_3,x_2^2x_3,x_2x_3^2,x_1x_2x_4,x_2^2x_4,x_2x_4^2,x_1x_3x_4,x_3^2x_4,x_3x_4^2,x_2x_3x_4)
.\] Then $I$ is a polymatroidal ideal such that $\dstab(I)=1$ and $\astab(I)=2$.
\end{Example}

The following example of matroidal ideal gives a negative answer to the Conjecture \ref{Co}.
\begin{Example}
Let $n=8$ and $I=(x_1x_2x_3x_4,x_1x_2x_3x_5,x_1x_2x_3x_6,x_1x_2x_3x_8,x_1x_2x_4x_7,\\x_1x_2x_5x_7,
x_1x_2x_6x_7,x_1x_2x_7x_8,x_1x_3x_4x_7,x_1x_3x_4x_8,x_1x_3x_5x_7,x_1x_3x_5x_8,x_1x_3x_6x_7,x_1x_3x_6x_8,\\x_1x_3x_7x_8,
x_1x_4x_7x_8,x_1x_5x_7x_8,x_1x_6x_7x_8,x_2x_3x_4x_5,x_2x_3x_4x_6,x_2x_3x_4x_7,x_2x_3x_5x_6,x_2x_3x_5x_7,\\x_2x_3x_5x_8,x_2x_3x_6x_7,x_2x_3x_6x_8,x_2x_3x_7x_8,x_2x_4x_5x_7,
x_2x_4x_6x_7,x_2x_5x_6x_7,x_2x_5x_7x_8,x_2x_6x_7x_8,\\x_3x_4x_5x_7,x_3x_4x_5x_8,x_3x_4x_6x_7,x_3x_4x_6x_8,x_3x_4x_7x_8,x_3x_5x_6x_7,x_3x_5x_6x_8,x_3x_5x_7x_8,
x_3x_6x_7x_8,\\x_4x_5x_7x_8,x_4x_6x_7x_8,x_5x_6x_7x_8)$.\\ Then $I$ is a matroidal ideal such that $\astab(I)=3$ and $\dstab(I)=2$.
\end{Example}

\begin{proof}
It is clear that $I$ is a matroidal ideal.
Since
$I^2:(x_2x_3x_7x_8)x_1x_5x_6=\frak{m}$, by Proposition \ref{P1} we conclude that $\dstab(I)=2$.
Also, by Corollary \ref{C1} we have $\astab(I)\leq 4$. Since $\Ass(I^2)\neq\Ass(I^3)=\Ass(I^4)$, it follows that $\astab(I)=3$.
\end{proof}



\end{document}